\documentclass[11pt]{amsart}

\newif\ifHideFoot
\HideFoottrue  
\HideFootfalse  

\ifHideFoot

\newcommand{\Yano}[1]{}
\newcommand{\Jeff}[1]{}
\newcommand{\Charles}[1]{}

\else 

\newcommand{\marg}[1]{\normalsize{{
			\color{red}\footnote{{\color{blue}#1}}}{\marginpar[\vskip
			-.25cm{\color{red}\hfill\tiny\thefootnote$\rightarrow$}]{\vskip
				-.2cm{\color{red}$\leftarrow$\tiny\thefootnote}}}}}
\newcommand{\Yano}[1]{\marg{(Yano) #1}}
\newcommand{\Jeff}[1]{\marg{(Jeff) #1}}
\newcommand{\Charles}[1]{\marg{(Charles) #1}}

\fi

\usepackage{mathpazo}
\usepackage{fullpage}
\usepackage[colorlinks]{hyperref}
\usepackage{amssymb,amsmath,amsthm,amscd,mathrsfs,graphicx, color}
\usepackage[cmtip,all,matrix,arrow,tips,curve]{xy}
\usepackage{hyperref}

\newtheorem{teo}{Theorem}[section]
\newtheorem{pro}[teo]{Proposition}

\newtheorem{cor}[teo]{Corollary}

\newtheorem{teoalpha}{Theorem}
\newtheorem*{teoalpha*}{Theorem}
\newtheorem*{coralpha*}{Corollary}
\newtheorem*{conalpha*}{Conjecture}

\newenvironment{alphabetize}{\begin{enumerate}

}{\end{enumerate}}

\theoremstyle{definition}

\theoremstyle{remark}
\newtheorem{rem}[teo]{Remark}

\DeclareMathOperator{\coniveau}{N}

\def\cx{{\mathbb C}}

\def\rat{{\mathbb Q}}
\def\integ{{\mathbb Z}}

\def\iso{\simeq}

\renewcommand{\bar}[1]{{\overline{#1}}}

\DeclareMathOperator{\alb}{Alb}

\DeclareMathOperator{\gal}{Gal}

\DeclareMathOperator{\pic}{Pic}

\DeclareMathOperator{\Ab}{Ab}
\DeclareMathOperator{\A}{A}

\def\red{{\rm red}}

\def\kbar{{\bar K}}

\title{Derived equivalent threefolds, algebraic representatives, and the coniveau filtration
}

\author{Jeffrey D. Achter}
\address{Colorado State University, Department of Mathematics,
	Fort Collins, CO 80523,
	USA}
\email{j.achter@colostate.edu}

\author{Sebastian Casalaina-Martin }
\address{University of Colorado, Department of Mathematics, 
	Boulder, CO 80309, USA }
\email{casa@math.colorado.edu}

\author{Charles Vial}
\address{Universit\"at Bielefeld, Germany}
\email{vial@math.uni-bielefeld.de}

\thanks{The first author was partially supported by  grants from the
	NSA (H98230-14-1-0161 and
	H98230-15-1-0247). 
	The second  author was partially supported by  grants from the
	Simons Foundation (317572) and the NSA (H98230-16-1-0053). 
	The third author was supported  by EPSRC Early Career Fellowship
	EP/K005545/1.}

\date{\today}

\begin{document}
	
	\maketitle

\begin{abstract}
	A conjecture of Orlov predicts that derived equivalent smooth
	projective varieties over a field have isomorphic Chow motives.
	The conjecture is known for curves, and was recently observed for
	surfaces by Huybrechts.  In this paper we focus on threefolds over
	perfect fields, and
	unconditionally secure results, which are implied by Orlov's
	conjecture, concerning  the geometric coniveau filtration, and abelian varieties attached to smooth projective varieties.
\end{abstract}

\section*{Introduction}
	
A fundamental invariant of a smooth projective variety $X$  is given by $D^b(X)$, the bounded derived category of coherent sheaves on $X$\,; see \cite{huybrechtsFM} for a survey. Two smooth projective varieties over a field $K$ are said to be \emph{derived equivalent} if there exists a $K$-linear exact 
equivalence of categories between their bounded derived categories of coherent sheaves. The following conjecture states that $D^b(X)$ is a finer invariant than the Chow motive of $X$ with rational coefficients.

\begin{conalpha*}[Orlov \cite{orlovmotives}]
 If two smooth projective varieties $X$ and $Y$ defined over a field $K$ are derived equivalent, then the Chow motives of $X$ and $Y$ with $\rat$-coefficients are isomorphic.
\end{conalpha*}

The conjecture is true for varieties $X$ and $Y$ with ample canonical or anti-canonical bundle, since then a fundamental theorem of Bondal--Orlov \cite{bondalorlov01} states that $X$ and $Y$ are isomorphic.  The conjecture is also well known for curves, 
 and  it has recently been observed by Huybrechts that the conjecture is true for surfaces \cite{huybrechts17}. 
In this paper we consider 
some consequences of the conjecture for threefolds. 

 As motivation, recall that 
a direct consequence of Orlov's conjecture is the following weaker
conjecture\,: \emph{if $X$ and $Y$ are smooth projective varieties
defined over a subfield $K\subseteq \cx$ that are derived equivalent,
then the Hodge structures $H^i(X_{\cx},\rat)$ and $H^i(Y_{\cx},\rat)$
are isomorphic for all $i$}. The case $i=0$ is true and obvious. The
first non-trivial result in this direction is due to Popa and Schnell,
building on work of Rouquier \cite{rouquier11}, who secure the case
$i=1$ with $K=\mathbb C$.
In fact, together with Honigs, we have established that their result is valid over an arbitrary field \cite{honigs3foldA}\,:

\begin{teoalpha}[Popa--Schnell  \cite{popaschnell}]
\label{T:ps}
 Assume that $X$ and $Y$ are two derived equivalent smooth projective
 varieties over a field $K$. Then the abelian varieties
 $\operatorname{Pic}^0(X)_{{\rm red}}$ and
 $\operatorname{Pic}^0(Y)_{{\rm red}}$ are isogenous. 
\end{teoalpha}
In particular, for all $\ell\ne
\operatorname{char}(K)$, the  $\operatorname{Gal}(K)$-representations
$H^1(X_{\bar K},\rat_\ell)$ and $H^1(Y_{\bar K},\rat_\ell)$ are
isomorphic.  In this paper we extend these results to other
Galois representations and abelian varieties attached to smooth
projective threefolds.

First, we observe that for any $i$  the $\ell$-adic \'etale
cohomology groups $H^i(-_{\bar K},\rat_\ell)$ of two derived equivalent
threefolds over an arbitrary field $K$ (of characteristic $\neq \ell$)
are isomorphic as $\operatorname{Gal}(K)$-modules (Proposition \ref{P:coniveau}).  Moreover, when $K$ is a perfect field, we show
 that these isomorphisms can be chosen to be compatible
with the geometric coniveau filtration (as reviewed in \S \ref{S:coniveau}), and
 we show that they share a common motivic
invariant, namely, the isogeny classes of their \emph{algebraic
  representatives} $\Ab^i(-)$. The algebraic representative
$\Ab^i(X)$ is an abelian variety over $K$ which is universal for
algebraically trivial cycles on $X$ of codimension $i$\,; see \S
\ref{S:algrep} for details.
 More precisely, our main result is\,: 
 
 \begin{teoalpha}\label{T:main}
 	Let $X$ and $Y$ be two smooth projective varieties of
        dimension $3$ over a perfect field $K$. Assume that $X$ and
        $Y$ are derived equivalent. 

\begin{alphabetize}

\item\label{galoisisom} For each nonnegative integer $i$, the Galois
  representations $H^i(X_{\bar K},\rat_\ell)$ and $H^i(Y_{\bar
  K},\rat_\ell)$ are isomorphic.  If $K\subseteq \mathbb C$, then the rational Hodge structures $H^i(X_{\mathbb C},\mathbb Q)$ and $H^i(Y_{\mathbb C},\mathbb Q)$ are isomorphic.  
  
\item If $K\subseteq \cx$, then there exist isomorphisms in
  (\ref{galoisisom}) compatible with the geometric  coniveau filtration (for both the $\ell$-adic and Betti cohomology).

\item\label{algrepisog} For each nonnegative integer $i$, the algebraic representatives $\Ab^i(X)$ and $\Ab^i(Y)$ are
  isogenous over $K$.
\end{alphabetize}

\end{teoalpha}

\noindent Insofar as $\Ab^1(X) \iso \pic^0(X)_\red$, part (\ref{algrepisog})
is a natural extension of Theorem \ref{T:ps} in the special case of
threefolds.
In fact,  in   part (\ref{algrepisog})
the case where  $i=3$ also follows directly from Theorem \ref{T:ps}.  In the language of Mazur's phantom abelian varieties \cite{mazurprobICCM, ACMVdcg}, a consequence of Theorem \ref{T:main} is that derived equivalent smooth projective uniruled threefolds over a field $K\subseteq \mathbb C$ have isogenous phantom abelian varieties.  

To a smooth projective variety $X$ over a subfield of $\cx$ one may
associate the total image of the Abel--Jacobi map restricted to algebraically trivial cycles, $J_a(X_{\mathbb C})$, an abelian variety defined over $K$
(\S \ref{S:intjac}, \cite{ACMVdcg2}) that sits inside the total intermediate Jacobian.   We show (Proposition \ref{P:totaljac}) that if $X$ and $Y$ are derived equivalent smooth projective varieties of arbitrary
dimension over a subfield of $\cx$, then $J_a(X_{\mathbb C})$ and $J_a(Y_{\mathbb C})$ are isogenous over $K$\,; over $\mathbb C$ this provides a short proof of a special case of a result  of  \cite{BT16}, which also considers semi-orthogonal decompositions. 

We also direct the reader to some related work  in the arithmetic setting.
Honigs  recently proved that two
derived equivalent surfaces \cite{honigs15} or  threefolds
\cite{honigs3fold} over a finite field share the same
zeta function. 
In another direction, Antieau, Krashen and Ward
\cite{antieaukrashenward} give examples of varieties $X$ and $Y$ over
a field $K$ that are derived equivalent, but are nontrivial twists of each
other.   For instance, they describe distinct $K$-isomorphism classes
of genus one curves that are derived equivalent.   These pairs of
curves are twists of the same elliptic curve, and so have isomorphic Jacobians, 
in accordance with  Theorem \ref{T:ps}.

\subsection*{Notation}
Given a field $K$, we denote by $\bar K$ a separable closure
and by $\gal(K) = \gal(\bar K/K)$ the absolute Galois
group of $K$.
For a smooth projective variety $X$ over $K$, $\operatorname{CH}(X)$ denotes the Chow group of $X$, and $H^i(X)(j)$ denotes one of the following Weil cohomology theories\,:
\begin{itemize}
	\item For prime $\ell \neq \operatorname{char}(K)$, $H^i(X)(j) = H^i_{\text{\'et}}(X_{\bar K},\rat_\ell(j))$ viewed as a $\rat_\ell[\operatorname{Gal}(K)]$-module.
	\item For $K\subseteq \cx$, $H^i(X)(j) = H^i(X_{\cx},\rat(j))$ viewed as a pure rational Hodge structure.
\end{itemize}
The notation is intentionally ambiguous so that we may give statements and proofs for both cohomology theories  simultaneously\,; the meaning will be clear from the context.   
Maps between $\rat_\ell[\operatorname{Gal}(K)]$-modules are always
assumed to be morphisms. Likewise, maps between rational Hodge
structures are always assumed to be morphisms.

\subsection*{Acknowledgments}
We thank Barry Mazur for suggesting we investigate the resonance between
phantom abelian varieties \cite{ACMVdcg} and derived equivalence,
Katrina Honigs and Daniel Huybrechts for helpful conversations, and the referee for helpful suggestions.

\section{Derived equivalence and the coniveau filtration}\label{S:coniveau}

If $X$ and $Y$ are two derived equivalent smooth projective varieties
over a field $K$, then we have (ungraded) isomorphisms  
(e.g., \cite[Prop.~5.33]{huybrechtsFM}, \cite[\S 2]{lieblicholsson},  \cite[Lem.~3.1]{honigs15})
\begin{equation}\label{eq:b}
\bigoplus_i H^{2i}(X)(i) \simeq \bigoplus_i H^{2i}(Y)(i) \quad \text{and} \quad \bigoplus_i  H^{2i+1}(X)(i) \simeq\bigoplus_i H^{2i+1}(Y)(i).
\end{equation}

Let $\coniveau^\bullet$ denote the (geometric) coniveau filtration, i.e., 
$$ \coniveau^j H^i(X) 
:= \sum \ker(
H^i(X)\rightarrow
H^i(X
\backslash Z)),$$
where the sum runs through all closed $K$-subschemes $Z\subseteq
X$ of codimension $\geq j$. The isomorphisms \eqref{eq:b} can be upgraded to take into account the coniveau filtration\,:

\begin{pro}\label{P:even-odd}
Let $X$ and $Y$ be smooth projective varieties over a perfect field $K$. Assume that $X$ and $Y$ are derived equivalent. Then there exist for all integers $j$ (ungraded) isomorphisms
\begin{equation}\label{eq:n}
\bigoplus_i \coniveau^{i-j}H^{2i}(X)(i) \simeq \bigoplus_i \coniveau^{i-j}H^{2i}(Y)(i) \ \text{and} \ \bigoplus_i  \coniveau^{i-j}H^{2i+1}(X)(i) \simeq\bigoplus_i \coniveau^{i-j}H^{2i+1}(Y)(i).
\end{equation}
\end{pro}
\begin{proof}
	By Orlov \cite{orlov}, the equivalence $D^b(X) \simeq D^b(Y)$ is induced by a Fourier--Mukai functor. Denote by $\mathcal{E} \in D^b(X\times_K Y)$ its kernel, and  by $\mathcal{F} \in D^b(Y\times_K X)$ the kernel of its inverse.
	To $\mathcal{E}$, one can associate the Mukai vector $$v(\mathcal E) := p_X^*\sqrt{\text{td}_X} \cdot \operatorname{ch}(\mathcal{E})\cdot p_Y^*\sqrt{\text{td}_Y} \in \operatorname{CH}(X\times_K Y) \otimes \mathbb Q,$$ 
where $p_X : X\times_K Y \rightarrow X$ and $p_Y : X\times_K Y \rightarrow Y$ are the natural projections, and where $\text{td}_X$ and $\text{td}_Y$ are the Todd classes of $X$ and $Y$, respectively. 
A square root is taken formally.  
Likewise, we may consider the Mukai vector $ v(\mathcal F)$, and we have (see e.g., \cite[Prop.~5.10, Lem.~5.32]{huybrechtsFM}) 
\begin{align*}
v(\mathcal F)\circ v(\mathcal E) &= v(\mathcal O_{\Delta_X}) = \Delta_X \in \operatorname{CH}(X\times_K X) \otimes \mathbb Q,\\
v(\mathcal E)\circ v(\mathcal F)& = v(\mathcal O_{\Delta_Y}) = \Delta_Y \in \operatorname{CH}(Y\times_K Y) \otimes \mathbb Q.
\end{align*}
It is then apparent that the action of $v(\mathcal E)$ induces the isomorphisms in \eqref{eq:b}. Assuming that $K$ is perfect, the isomorphisms in \eqref{eq:n} follow readily from the functoriality of the coniveau filtration with respect to the action of correspondences \cite{arapura-kang, deglise}.
\end{proof}

\begin{rem}\label{R:GHC} Proposition \ref{P:even-odd} also holds if one replaces the geometric coniveau filtration with the Hodge coniveau filtration or the Tate coniveau filtration\,; see e.g. \cite[\S 1.2]{ACMVdcg} for definitions. As a consequence, given two derived equivalent smooth projective varieties $X$ and $Y$ over a field $K$, if $K=\cx$ and the Hodge conjecture (or the generalized Hodge conjecture) holds for $X$, or if $K$ is finitely generated over its prime subfield and the Tate conjecture (or the generalized Tate conjecture) holds for $X$, then the corresponding  conjecture  also   holds for $Y$.
\end{rem}

The conjecture of Orlov predicts that the isomorphisms \eqref{eq:n} can be chosen to be $i$-graded. We verify this prediction in the case where $X$ and $Y$ are derived equivalent threefolds. Note that it was already observed by Honigs \cite{honigs3fold} that the theorem of Popa--Schnell implies in this case that $H^3(X)\simeq H^3(Y)$ (it was previously observed in \cite{popaschnell} that $X$ and $Y$ have the same Betti 
numbers).

\begin{pro}\label{P:coniveau}
	Let $X$ and $Y$ be two smooth projective varieties of dimension $3$ over a field $K$. Assume that $X$ and $Y$ are derived equivalent. 
	Then for all $i$ there are isomorphisms 
	$H^i(X) \simeq H^i(Y)$, which  are compatible with the coniveau filtration for odd $i$ if $K$ is assumed to be perfect, and for all $i$ if $K\subseteq \mathbb C$.  
\end{pro}

\begin{proof} 
	Note that for a 3-fold $V$,  we have 
	$$0= \coniveau^1H^0(V) \subseteq \coniveau^0H^0(V) = H^0(V)$$
	$$0= \coniveau^1H^1(V)   \subseteq \coniveau^0H^1(V) = H^1(V)$$
	$$0=  \coniveau^2H^2(V)   \subseteq \coniveau^1H^2(V) \subseteq  \coniveau^0H^2(V) = H^2(V)$$
	$$0= \coniveau^2H^3(V)   \subseteq \coniveau^1H^3(V) \subseteq  \coniveau^0H^3(V) = H^3(V)$$
	$$0= \coniveau^3H^4(V)   \subseteq \coniveau^2H^4(V) \subseteq  \coniveau^1H^4(V) = H^4(V)$$
	$$0= \coniveau^3H^5(V)   \subseteq \coniveau^2H^5(V) = H^5(V)$$
	$$0= \coniveau^4H^6(V)   \subseteq \coniveau^3H^6(V) = H^6(V).$$
The equalities on the left follow from dimension considerations, while those on the right follow from the Lefschetz hyperplane theorem.

The category of polarizable rational Hodge structures is semisimple,
and in particular has the cancellation property; if there is an
isomorphism $A\oplus B \simeq A \oplus B'$ of such Hodge structures,
then $B\simeq B'$.  The category of finite-dimensional
$\rat_\ell$-representations of $\gal(K)$ (equivalently,
$\rat_\ell[\gal(K)]$-modules with finite-dimensional underlying
$\rat_\ell$-vector space) also has the cancellation
property.  Indeed, any such representation satisfies both the ascending and
descending chain conditions, and thus (e.g., \cite[Thm.~6.12,
Cor.~6.15]{curtisreiner}) satisfies the hypotheses of the
Krull--Schmidt--Azumaya theorem.

	That $H^0(X) \simeq H^0(Y)$ and $H^6(X)\simeq H^6(Y)$ is
        obvious.  Theorem \ref{T:ps} asserts that $H^1(X)\simeq
        H^1(Y)$ and duality then implies that $H^5(X)\simeq H^5(Y)$. By \eqref{eq:b} and cancellation, we obtain $H^3(X)\simeq H^3(Y)$. We also obtain from \eqref{eq:n} and from cancellation that $\coniveau^1H^3(X)\simeq \coniveau^1H^3(Y)$.

	By \eqref{eq:b}, and by
	cancellation, we obtain that $H^2(X)(1) \oplus H^4(X)(2) \simeq
	H^2(Y)(1) \oplus H^4(Y)(2)$.
	By Poincar\'e duality, and by
	semi-simplicity of the category of polarizable $\rat$-Hodge
	structures in the case where $H$ is Betti cohomology or by the Krull--Schmidt--Azumaya theorem  in the case where $H$ is $\ell$-adic cohomology, we obtain that $H^2(X)  \simeq
	H^2(Y)$, and then by duality that $H^4(X)  \simeq
	H^4(Y)$.
	
	We now assume  $K\subseteq \cx$.  	 Recall that in this case the coniveau filtration on $H^i(X)$ is split\,; see e.g., \cite[Cor.~4.4]{ACMVdcg2}. Therefore, in order to prove isomorphisms of $\rat_\ell[\gal(K)]$-modules or of rational Hodge structures that are compatible with the coniveau filtration, it suffices to prove that $\coniveau^jH^i(X) \simeq \coniveau^jH^i(Y)$ for all $i$ and $j$.
	It only remains to show that $\coniveau^1 H^2(X) \simeq \coniveau^1 H^2(Y)$ and $\coniveau^2 H^4(X) \simeq \coniveau^2 H^4(Y)$. 
 Using  Proposition \ref{P:even-odd} for $j=0$ we have that   
$$
\coniveau^0 H^0 \oplus \coniveau^1 H^2 (1)\oplus \coniveau ^2 H^4(2)
\oplus \coniveau ^3 H^6 (3) 
$$	
is a derived invariant.	
	Note that by the comparison isomorphisms (see e.g., \cite[(1.3)]{ACMVdcg}), intersecting with a smooth hyperplane section  defined over $K$ induces an isomorphism $\coniveau^1 H^{2}(1) \simeq \coniveau^2 H^{4}(2)$. 
 Now using  semi-simplicity or  the Krull--Schmidt theorem, we obtain the required isomorphisms for $\coniveau^1H^2$ and $\coniveau^2H^4$, which concludes the proof.
\end{proof}

\begin{rem}\label{R:coniveau}
The same arguments show that if $X$ and $Y$ are derived equivalent smooth projective varieties of dimension 4 over a field $K$, then there are isomorphisms $H^3(X) \simeq H^3(Y)$ and $H^5(X)\simeq H^5(Y)$. If in addition the monomorphism $\coniveau^1H^3(X) \hookrightarrow \coniveau^2H^5(X)$ induced by cupping with an ample divisor is surjective (e.g., if the generalized Hodge (or Tate) conjecture holds for $X$, or if the standard conjectures hold for $X$), then there are isomorphisms $\coniveau^1H^3(X) \simeq \coniveau^1H^3(Y)$ and $\coniveau^2H^5(X)\simeq \coniveau^2H^5(Y)$ compatible with the aforementioned isomorphisms.
\end{rem}

\section{Derived equivalence and total intermediate Jacobians}
\label{S:intjac}

If $K\subseteq \cx$, we see from the second isomorphism of \eqref{eq:b} that the isogeny class of the total intermediate Jacobian $$J(X_\cx) := \bigoplus_i J^{2i-1}(X_\cx)$$ is a derived invariant of smooth projective varieties. Here, $J^{2i-1}(X_\cx)$ is Griffiths' intermediate Jacobian\,; as a complex torus it is defined as 
\[
J^{2i-1}(X) := F^{i} H^{2i-1}(X,\cx) \backslash H^{2i-1}(X,\cx) /
H^{2i-1}(X,\integ),
\] 
where $F^\bullet$ denotes the Hodge filtration. The intermediate
Jacobian is of interest because it receives cohomologically trivial
cycles under the Abel--Jacobi map 
$
AJ:\operatorname{CH}^i(X_\cx)_{{\rm hom}} \to J^{2i-1}(X_\cx)
$.   In this paper we will be interested in the   Abel--Jacobi map  
\begin{equation}\label{E:AJmap}
\xymatrix{
AJ:\operatorname{A}^i(X_\cx) \ar[r] & J^{2i-1}(X_\cx)
}
\end{equation}
obtained by restricting to the subgroup of algebraically trivial cycles  $\A^i(X_\cx) \subseteq \operatorname{CH}^i(X_\cx)_{{\rm hom}}$.
We denote by 
$J^{2i-1}_a(X_\cx)$  the subtorus  that is the  image of $\operatorname{A}^i(X_{\mathbb C})$ under the  Abel--Jacobi map  (e.g., \cite[Lem.~1.6.2]{murre83}).
In terms of the coniveau filtration, we have 
\[
H^1(J_a^{2i-1}(X),\rat) \simeq \coniveau^i H^{2i-1}(X,\rat(i)).
\] 
A choice of polarization on $X$ endows $J^{2i-1}_a(X_\cx)$ with a polarization, thereby turning it into a complex abelian variety. 
Likewise, we see from the second isomorphism of \eqref{eq:n} that the isogeny class of the total image of the Abel--Jacobi map \eqref{E:AJmap},
 $$J_a(X_\cx) := \bigoplus_i J_a^{2i-1}(X_\cx),$$ is a derived invariant of smooth projective varieties.

We would like to upgrade this observation to an isogeny of abelian
varieties defined over $K$. Recall from \cite{ACMVdcg2} that $J^{2i-1}_a(X_\cx)$ descends canonically to
an abelian variety $J^{2i-1}_a(X)$ over $K$. Precisely, $J^{2i-1}_a(X)$ is the
abelian variety over $K$ such that the Abel--Jacobi map
$\operatorname{A}^i(X_\cx) \to J^{2i-1}_a(X)_\cx$ is
$\operatorname{Aut}(\cx/K)$-equivariant.

\begin{pro}\label{P:totaljac}
	Let $X$ and $Y$ be smooth projective varieties over a
        subfield $K$ of $\cx$. Assume that $X$ and $Y$ are derived equivalent. Then   $J_a(X) := \bigoplus_i J_a^{2i-1}(X)$ and $J_a(Y) := \bigoplus_i J_a^{2i-1}(Y)$ are isogenous abelian varieties over $K$. 
\end{pro}

\begin{proof}
	Recall that the Abel--Jacobi map is functorial with respect to the action of correspondences. 
	With the notations of the proof of Proposition \ref{P:even-odd}, the Mukai vector $v(\mathcal E)$ thus induces an isogeny $v(\mathcal E)_* : J_a(X_\cx) \to J_a(Y_\cx)$. Then it follows from \cite[Prop.~5.1]{ACMVdcg2} that this isogeny descends to an isogeny over $K$.
\end{proof}

As mentioned earlier, over $\mathbb C$ this provides a short proof of a special case of a result  of  \cite{BT16}, which also considers semi-orthogonal decompositions.

\vskip .2 cm 

For threefolds, as a corollary to Proposition \ref{P:totaljac} (Proposition \ref{P:coniveau} would suffice if $K=\cx$) and  Theorem \ref{T:ps},  we may single out the second intermediate Jacobian (we will give an alternate proof of this corollary in Theorem \ref{T:mainAb})\,:

\begin{cor}\label{C:algjac}
	Let $X$ and $Y$ be two smooth projective varieties of dimension $3$ over a field $K \subseteq \cx$. Assume that $X$ and $Y$ are derived equivalent. Then the intermediate Jacobians $J^3(X_\cx)$ and $J^3(Y_\cx)$ are isogenous complex tori, and    $J_a^3(X)$ and $J_a^3(Y)$ are $K$-isogenous  abelian varieties.\qed
\end{cor}

\begin{rem}
	Following Remarks \ref{R:GHC} and \ref{R:coniveau}, we may in fact improve Corollary \ref{C:algjac}. Suppose that $X$ and $Y$ are two derived-equivalent smooth projective varieties of dimension $4$ over a field $K \subseteq \cx$\,; then the intermediate Jacobians $J^{2i-1}(X_\cx)$ and $J^{2i-1}(Y_\cx)$ are isogenous complex tori for all $i$. If in addition the monomorphism $\coniveau^1H^3(X) \hookrightarrow \coniveau^2H^5(X)$ induced by cupping with an ample divisor is surjective, then  
  $J_a^{2i-1}(X)$ and $J_a^{2i-1}(Y)$ are $K$-isogenous  abelian varieties for all $i$.
\end{rem}

\section{Derived equivalent 3-folds and algebraic representatives}\label{S:algrep}

The aim of this section is to extend Corollary \ref{C:algjac} to varieties defined over a perfect field $K$. The abelian variety that plays the role of the second   intermediate Jacobian over algebraically closed fields of positive characteristic is called an \emph{algebraic representative} (for codimension-$2$ cycles).

Let $X/K$ be a smooth,
projective variety, and consider the group $\A^i(X_\kbar)$ of algebraically
trivial codimension-$i$ cycles up to rational equivalence on $X_\kbar$.
The \emph{algebraic representative}, if it exists, is an abelian
variety $\Ab^i(X_\kbar)$ equipped with an Abel--Jacobi map
$\phi^i: \A^i(X_\kbar) \to \Ab^i(X_\kbar)(\kbar)$. Its defining universal property
is that $\phi^i$ is initial among all \emph{regular homomorphisms}
$\operatorname{A}^i(X_\kbar) \to A(\kbar)$, that is, among all homomorphisms
of groups $\psi : \operatorname{A}^i(X_\kbar) \to A(\kbar)$ to the $\kbar$-points
of an abelian variety $A$ over $\kbar$ such that for every pair
$((T,t_0),Z)$ with $(T,t_0)$ a pointed smooth integral variety over
$\kbar$, and $Z\in \operatorname{CH}^i(T\times X)$, the composition
$T(\kbar) \to \operatorname{A}^i(X_\kbar) \to A(\kbar), t \mapsto \psi(Z_t -
Z_{t_0}) $ is induced by a morphism of varieties $\psi_Z:T\to A$ over
$\kbar$.  The Abel--Jacobi map to the algebraic
representative is unique (up to a unique automorphism) and surjective.
In particular, $\Ab^1(X_\kbar) \iso \pic^0(X_\kbar)_{{\rm red}}$, while
$\Ab^{\dim X}(X_\kbar) \iso \alb(X_\kbar)$.  The existence of $\Ab^2(X_\kbar)$ was
established by Murre \cite{murre83}.

In \cite{ACMVdcg}, we extended Murre's theorem 
 by showing that an initial regular
homomorphism $$\phi_X^i : \operatorname{A}^i(X_\kbar) \to \operatorname{Ab}^i(X_\kbar)(\kbar)$$
can be made Galois-equivariant, thus providing a canonical descent
datum on $\operatorname{Ab}^i(X_\kbar)$ and hence, by descent,   a
distinguished model $\operatorname{Ab}^i(X)$ over $K$ of the abelian
variety $\operatorname{Ab}^i(X_\kbar)$.  The corresponding statements for
the Picard and Albanese varieties are well-known.\medskip

Our main result extends Theorem \ref{T:ps} to
$\operatorname{Ab}^2(X)$ when $\dim X \leq 3$ and when $K$ is a perfect field.

\begin{teo} \label{T:mainAb}
	Let $X$ and $Y$ be two smooth projective varieties of dimension $3$ over a perfect field $K$. Assume that $X$ and $Y$ are derived equivalent. Then the abelian varieties $\operatorname{Ab}^2(X)$ and $\operatorname{Ab}^2(Y)$ are $K$-isogenous.
\end{teo}
\begin{proof}
	By Orlov \cite{orlov}, the equivalence $D^b(X) \simeq D^b(Y)$ is induced by a Fourier--Mukai functor. Denote $\mathcal{E} \in D^b(X\times_K Y)$ its kernel.
	As above, to $\mathcal{E}$  one can associate the Mukai vector $$\gamma = v(\mathcal E) := p_X^*\sqrt{\text{td}_X} \cdot \operatorname{ch}(\mathcal{E})\cdot p_Y^*\sqrt{\text{td}_Y} \in \operatorname{CH}^*(X\times_K Y) \otimes \mathbb Z [\frac{1}{2}, \frac{1}{3}, \frac 15].$$ 
	(Here, $2,3$ and $5$ are the only prime numbers appearing as factors of denominators of the Chern character and of the square root of Todd classes.)
Choose a positive integer
	$n$,
	whose only prime divisors are $2$, $3$ and $5$,
	such that $n\gamma$ is an integral cycle\,;
	the correspondence $n\gamma \in\operatorname{CH}(X\times_K Y)$ induces a homomorphism  $$n \gamma_* :  \operatorname{A}^*(X_{\kbar}) \longrightarrow
	\operatorname{A}^*(Y_{\kbar})$$
	which, by Riemann--Roch, becomes an isomorphism upon tensoring with $\mathbb Z[1/30]$.
	
	 Since $n\gamma_*$ is induced by an integral cycle,	 the homomorphism $(\phi^3_Y\oplus \phi^2_Y\oplus \phi^1_Y) \circ n\gamma_*$ is a regular homomorphism.  
	Thus, for $X$ and $Y$ as in the statement of the theorem and
        by the universal property of the algebraic representatives, we
        have a diagram
	\begin{equation}\label{E:diag}
	\xymatrix@C=5em@R=.3em{
		\operatorname{A}^*(X_\kbar) \ar@{->>}[r]^<>(0.5){\phi^3_X\oplus \phi^2_X\oplus \phi^1_X} \ar[ddd]_{n\gamma_*} & \operatorname{Alb}(X)(\kbar) \oplus \operatorname{Ab}^2(X)(\kbar) \oplus \operatorname{Pic}^0(X)_\red(\kbar) \ar[ddd]_\varphi \\
		\\
		\\
		\operatorname{A}^*(Y_\kbar) \ar@{->>}[r]^<>(0.5){\phi^3_Y\oplus \phi^2_Y\oplus \phi^1_Y}& \operatorname{Alb}(Y)(\kbar) \oplus \operatorname{Ab}^2(Y)(\kbar) \oplus \operatorname{Pic}^0(Y)_\red(\kbar)
	}
	\end{equation}
	where $\varphi$ is induced by a $\kbar$-homomorphism $$\operatorname{Alb}(X)_\kbar \times \operatorname{Ab}^2(X)_\kbar \times \operatorname{Pic}^0(X)_{\red,\kbar} \to \operatorname{Alb}(Y)_\kbar \times \operatorname{Ab}^2(Y)_\kbar \times \operatorname{Pic}^0(Y)_{\red,\kbar}.$$
	In fact, since the integral correspondence $n\gamma$ is defined over $K$,  it follows from \cite[Thm.~4.4]{ACMVdcg} that this $\kbar$-homomorphism descends to a $K$-homomorphism 
	$$\Phi : \operatorname{Alb}(X) \times_K \operatorname{Ab}^2(X) \times_K \operatorname{Pic}^0(X)_\red \to \operatorname{Alb}(Y) \times_K \operatorname{Ab}^2(Y) \times_K \operatorname{Pic}^0(Y)_\red.$$
	Together with the fact that $n\gamma_*$ is an isomorphism after inverting $30$,  a simple diagram chase yields that $\varphi$ is surjective after inverting $30$ and hence that $\Phi$ is surjective.

	Likewise, by considering the inverse equivalence, we see that there is a  $K$-homomorphism 
	$$\Psi : \operatorname{Alb}(Y) \times_K \operatorname{Ab}^2(Y) \times_K \operatorname{Pic}^0(Y)_\red \to \operatorname{Alb}(X) \times_K \operatorname{Ab}^2(X) \times_K \operatorname{Pic}^0(X)_\red$$
	which is surjective.
	Therefore
	the abelian varieties $\operatorname{Alb}(X) \times_K \operatorname{Ab}^2(X) \times_K \operatorname{Pic}^0(X)_\red$ and $\operatorname{Alb}(Y) \times_K \operatorname{Ab}^2(Y) \times_K \operatorname{Pic}^0(Y)_\red$ are $K$-isogenous.
	
	By Theorem \ref{T:ps}, if
	$X$ and $Y$ are two derived equivalent smooth projective varieties
	over a field $K$, then the abelian varieties $\operatorname{Pic}^0(X)_\red$
	and $\operatorname{Pic}^0(Y)_\red$ are $K$-isogenous\,; by duality, the
	abelian varieties $\operatorname{Alb}(X)$ and $\operatorname{Alb}(Y)$
	are $K$-isogenous. 
	Therefore, by Poincar\'e reducibility,
	 we find that $\operatorname{Ab}^2(X) $ and 
	 $\operatorname{Ab}^2(Y)$ are $K$-isogenous.
\end{proof}

\bibliographystyle{hamsalpha}
\bibliography{DCG}

\end{document}